\documentclass{amsart}
\usepackage[all,ps,cmtip]{xy}
\usepackage{amsmath, amssymb}
\usepackage{amscd}
\usepackage{tikz,color}

\newtheorem{theorem}{Theorem}[section]
\newtheorem{lemma}[theorem]{Lemma}

\newtheorem{proposition}[theorem]{Proposition}
\newtheorem{conjecture}[theorem]{Conjecture}
\newtheorem{assumption}[theorem]{Assumption}

\theoremstyle{definition}
\newtheorem{definition}[theorem]{Definition}

\theoremstyle{remark}

\numberwithin{equation}{section}

\DeclareMathOperator{\rank}{rk}

\DeclareMathOperator{\Hom}{Hom}

\DeclareMathOperator{\Coh}{Coh}

\DeclareMathOperator{\D}{D}

\DeclareMathOperator{\ch}{ch}

\begin{document}
\title{Arithmetic genus of integral space curves}

\author{Hao Max Sun}

\address{Department of Mathematics, Shanghai Normal University, Shanghai 200234, People's Republic of China}

\email{hsun@shnu.edu.cn, hsunmath@gmail.com}



\date{June 6, 2016}


\begin{abstract}
We give an estimation for the arithmetic genus of an integral space
curve, which are not contained in a surface of degree $k-1$. Our
main technique is the Bogomolov-Gieseker type inequality for
$\mathbb{P}^3$ proved by Macr\`i.
\end{abstract}

\maketitle

\setcounter{tocdepth}{1}

\section{Introduction}
A classical problem, which goes back to Halphen \cite{H}, is to
determine, for given integers $d$ and $k$, the maximal genus
$G(d,k)$ of a smooth projective space curve of degree $d$ not
contained in a surface of degree $< k$. This problem is actually
very natural, and has been investigated by many people (see
\cite{Cast, GP, Har, Ha, HH}).

In this paper, we consider the same problem for an integral space
curve. Our main result is:

\begin{theorem}\label{main}
Let $C$ be an integral complex projective curve in $\mathbb{P}^3$ of
degree $d$. Let $p_a(C)$ be its arithmetic genus. If $C$ is not
contained in a surface of degree $< k$. Then
$$p_a(C)\leq
\left\{
\begin{array}{lcl}
\frac{2d^2}{3k}+\frac{1}{3}d(k-6)+1,  & &\mbox{if}~k^2<d,\\
&&\\
d(\sqrt{d}-2)+1, & &\mbox{if}~k^2\geq d.
\end{array}\right.
$$
\end{theorem}
For the case of $k\leq2$, this inequality has also been obtained by
Macr\`i \cite[Corollary 4.1]{Mac2}. When $k^2<d$, our bound is
weaker than that of Castelnuovo, Harris and Gruson-Peskine for
smooth space curves, but still non-trivial. Our bound can be reached
in some cases, when $k^2\geq d$. For example, the arithmetic genus
of a complete intersection of two surfaces of degree $k$ is
$$k^2(k-2)+1=d(\sqrt{d}-2)+1.$$

The idea of the proof of Theorem \ref{main} is to establish the
tilt-stability of $\mathcal{I}_C$ via computing its walls, then the
Bogomolov-Gieseker type inequality for $\mathbb{P}^3$ proved by
Macr\`i \cite{Mac2} implies Theorem \ref{main}. This
Bogomolov-Gieseker type inequality naturally appears in the
construction of Bridgeland stability conditions on threefolds (cf.
\cite{Bri1, BMT, BMS}).  There are also some other interesting
applications of the Bogomolov-Gieseker type inequality in
\cite{BBMT} and \cite{Toda}.

Our tilt-stability of $\mathcal{I}_C$ can also gives a version of
the Halphen Speciality Theorem:

\begin{theorem}\label{main2}
Let $C\subset \mathbb{P}^3$ be an integral complex projective degree
$d$ curve not contained in any surface of degree $<k$. Then
$h^2(\mathcal{I}_C(l))=h^1(\mathcal{O}_C(l))=0$, if
$l>\frac{2d}{k}-4$ when $k^2<d$, or $l>2\sqrt{d}-4$ when $k^2\geq
d$.
\end{theorem}

Our paper is organized as follows. In Section \ref{S2}, we review
basic properties of tilt-stability, the conjectural inequality
proposed in \cite{BMT, BMS} and variants of the classical
Bogomolov-Gieseker inequality satisfies by tilt-stable objects. Then
in Section \ref{S3} the tilt-stability of $\mathcal{I}_C$ has been
established via computing its walls. Finally, we show the proof of
Theorem \ref{main} and \ref{main2} in Section \ref{S4}.

\subsection*{Notation}
In this paper, we will always denote by $C$ an integral projective
curve in the three dimensional complex projective space
$\mathbb{P}^3$ and by $\mathcal{I}_C$ its ideal sheaf in
$\mathbb{P}^3$. We let $p_a(C):=h^1(C, \mathcal{O}_C)$ be the
arithmetic genus of $C$. By $X$ we denote a complex smooth
projective threefold and by $\D^b(X)$ its bounded derived category
of coherent sheaves.

\subsection*{Acknowledgments}
The author would like to thank Emanuele Macr\`i for useful
discussions.

\section{Preliminaries}\label{S2}
In this section, we review the notion of tilt-stability for
threefolds introduced in \cite{BMT, BMS}. Then we recall the
Bogomolov-Gieseker type inequality for tilt-stable complexes
proposed there.

Let $X$ be a smooth projective threefold over $\mathbb{C}$, and let
$H$ be an ample divisor on $X$. Let $\alpha>0$ and $\beta$ be two
real numbers. We write $\ch^{\beta}(E)=e^{-\beta H}\ch(E)$ denotes
the Chern character twisted by $\beta H$. More explicitly, we have
\begin{eqnarray*}
\begin{array}{lcl}
\ch^{\beta}_0=\ch_0=\rank  && \ch^{\beta}_2=\ch_2-\beta H\ch_1+\frac{\beta^2}{2}H^2\ch_0\\
&&\\
\ch^{\beta}_1=\ch_1-\beta H\ch_0 && \ch^{\beta}_3=\ch_3-\beta
H\ch_2+\frac{\beta^2}{2}H^2\ch_1-\frac{\beta^3}{6}H^3\ch_0.
\end{array}
\end{eqnarray*}

\subsection*{Slope-stability}
We define the slope $\mu_{\beta}$ of a coherent sheaf $E\in \Coh(X)$
by
\begin{eqnarray*}
\mu_{\beta}(E)= \left\{
\begin{array}{lcl}
+\infty,  & &\mbox{if}~\ch^\beta_0(E)=0,\\
&&\\
\frac{H^2\ch_1^{\beta}(E)}{H^3\ch_0^{\beta}(E)}, &
&\mbox{otherwise},
\end{array}\right.
\end{eqnarray*}

\begin{definition}
A coherent sheaf $E$ on $X$ is slope-(semi)stable (or
$\mu_{\beta}$-(semi)stable) if, for all non-zero subsheaves
$F\hookrightarrow E$, we have
$$\mu_{\beta}(F)<(\leq)\mu_{\beta}(E/F).$$
\end{definition}
Harder-Narasimhan filtrations (HN-filtrations, for short) with
respect to slope-stability exist in $\Coh(X)$: given a non-zero
sheaf $E\in\Coh(X)$, there is a filtration
$$0=E_0\subset E_1\subset\cdots\subset E_n=E$$
such that: $G_i:=E_i/E_{i-1}$ is slope-semistable, and
$\mu_{\beta}(G_1)>\cdots>\mu_{\beta}(G_n)$. We set
$\mu^+_{\beta}(E):=\mu_{\beta}(G_1)$ and
$\mu^-_{\beta}(E):=\mu_{\beta}(G_n)$.

\subsection*{Tilt-stability}
There exists a \emph{torsion pair}
$(\mathcal{T}_{\beta},\mathcal{F}_{\beta})$ in $\Coh(X)$ defined as
follows:
\begin{eqnarray*}
\mathcal{T}_{\beta}=\{E\in\Coh(X):\mu^-_{\beta}(E)>0 \}\\
\mathcal{F}_{\beta}=\{E\in\Coh(X):\mu^+_{\beta}(E)\leq0 \}
\end{eqnarray*}
Equivalently, $\mathcal{T}_{\beta}$ and $\mathcal{F}_{\beta}$ are
the extension-closed subcategories of $\Coh(X)$ generated by
slope-stable sheaves of positive and non-positive slope,
respectively.

\begin{definition}
We let $\Coh^{\beta}(X)\subset \D^b(X)$ be the extension-closure
$$\Coh^{\beta}(X)=\langle\mathcal{T}_{\beta}, \mathcal{F}_{
\beta}[1]\rangle.$$
\end{definition}

By the general theory of torsion pairs and tilting \cite{HRS},
$\Coh^{\beta}(X)$ is the heart of a bounded t-structure on
$\D^b(X)$; in particular, it is an abelian category.

Now we can define the following slope function on $\Coh^{
\beta}(X)$: for an object $E\in\Coh^{\beta}(X)$, we set

\begin{eqnarray*}
\nu_{\alpha, \beta}(E)= \left\{
\begin{array}{lcl}
+\infty,  & &\mbox{if}~H^2\ch^{\beta}_1(E)=0,\\
&&\\
\frac{H\ch_2^{\beta}(E)-\frac{1}{2}\alpha^2H^3\ch^{\beta}_0(E)}{H^2\ch^{\beta}_1(E)},
& &\mbox{otherwise}.
\end{array}\right.
\end{eqnarray*}

\begin{definition}
An object $E\in\Coh^{\beta}(X)$ is \emph{tilt-(semi)stable} (or
$\nu_{\alpha,\beta}$-\emph{(semi)stable}) if, for all non-trivial
subobjects $F\hookrightarrow E$, we have
$$\nu_{\alpha, \beta}(F)<(\leq)\nu_{\alpha, \beta}(E/F).$$
\end{definition}

Lemma 3.2.4 in \cite{BMT} shows that the Harder-Narasimhan property
holds with respect to $\nu_{\alpha,\beta}$-stability, i.e., for any
$\mathcal{E}\in\Coh^{\beta}(X)$ there is a filtration in
$\Coh^{\beta}(X)$
$$0=\mathcal{E}_0\subset \mathcal{E}_1\subset\cdots\subset \mathcal{E}_n=\mathcal{E}$$ such that:
$\mathcal{F}_i:=\mathcal{E}_i/\mathcal{E}_{i-1}$ is
$\nu_{\alpha,\beta}$-semistable with
$\nu_{\alpha,\beta}(\mathcal{F}_1)>\cdots>\nu_{\alpha,\beta}(\mathcal{F}_n)$.

\begin{definition}
In the above filtration, we call $\mathcal{E}_1$ the
$\nu_{\alpha,\beta}$-maximal subobject of
$\mathcal{E}\in\Coh^{\beta}(X)$. If $\mathcal{E}$ is
$\nu_{\alpha,\beta}$-semistable, we say $\mathcal{E}$ itself to be
its $\nu_{\alpha,\beta}$-maximal subobject.
\end{definition}

\subsection*{Bogomolov-Gieseker type inequality}We now recall the
Bogomolov-Gieseker type inequality for tilt-stable complexes
proposed in \cite{BMT, BMS}.
\begin{definition}
We define the generalized discriminant
$$\overline{\Delta}^{\beta}_H:=(H^2\ch^{\beta}_1)^2-2H^3\ch^{\beta}_0\cdot(H\ch^{\beta}_2).$$
\end{definition}
A short calculation shows
$\overline{\Delta}^{\beta}_H=(H^2\ch_1)^2-2H^3\ch_0\cdot(H\ch_2)$.
Hence the generalized discriminant is independent of $\beta$.

\begin{theorem}[{\cite[Theorem 7.3.1]{BMT}}]\label{2.6}
Assume $E\in\Coh^{\beta}(X)$ is $\nu_{\alpha,\beta}$-semistable.
Then
\begin{equation}\label{BG}
\overline{\Delta}^{\beta}_H(E)\geq0.
\end{equation}
\end{theorem}

\begin{conjecture}[{\cite[Conjecture 4.1]{BMS}}]
Assume $E\in\Coh^{\beta}(X)$ is $\nu_{\alpha,\beta}$-semistable.
Then
\begin{equation}\label{BG}
\alpha^2\overline{\Delta}^{\beta}_H(E)+4\left(H\ch^{\beta}_2(E)\right)^2-6H^2\ch^{\beta}_1(E)\ch^{\beta}_3(E)\geq0.
\end{equation}
\end{conjecture}

Such an inequality was proved by Macr\`i \cite{Mac2} in the case of
the projective space $\mathbb{P}^3$:
\begin{theorem}\label{*}
The inequality (\ref{BG}) holds for $\nu_{\alpha,\beta}$-semistable
objects in $\D^b(\mathbb{P}^3)$.
\end{theorem}

\section{Tilt-stability of ideal sheaves of space curves}\label{S3}
In this section, we establish the tilt-stability of ideal sheaves of
spaces curves via computing their walls. Then from Theorem
(\ref{*}), we can deduce a Castelnuovo type inequality for integral
curves in $\mathbb{P}^3$.

Throughout this section, let $C$ be an integral projective curve in
$\mathbb{P}^3$ of degree $d$ not contained in a surface of degree $<
k$, and let $\mathcal{I}_C$ be the ideal sheaf of $C$ in
$\mathbb{P}^3$. We keep the same notation as that in the previous
section for $X=\mathbb{P}^3$ and $H= \mbox{a plane of}~
\mathbb{P}^3$. To simplify, we directly identify
$H^{3-i}\ch^{\beta}_i(E)=\ch^{\beta}_i(E)$ for
$E\in\D^b(\mathbb{P}^3)$. The tilted slope becomes:
$$\nu_{\alpha,\beta}=\frac{\ch^{\beta}_2-\frac{1}{2}\alpha^2\ch^{\beta}_0}{\ch^{\beta}_1}=\frac{\ch_2-\beta\ch_1+\frac{1}{2}(\beta^2-\alpha^2)\ch_0}{\ch_1-\beta\ch_0}.$$

The following lemma is a key observation for us to establish the
tilt-stability of $\mathcal{I}_C$.
\begin{lemma}\label{lemma1}
Let $E$ be the $\nu_{\alpha,\beta}$-maximal subobject of
$\mathcal{I}_C\in \Coh^{\beta}(\mathbb{P}^3)$ for some
$(\alpha,\beta)\in \mathbb{R}_{>0}\times \mathbb{R}$. If
$2\alpha^2+\beta^2\geq4d$, then $\ch_0(E)=1$.
\end{lemma}
\begin{proof}
By the long exact sequence of cohomology sheaves induced by the
short exact sequence
$$0\rightarrow E\rightarrow \mathcal{I}_C\rightarrow
Q\rightarrow0$$ in $\Coh^{\beta}(\mathbb{P}^3)$, one sees that $E$
is a torsion free sheaf with $\ch_0(E)\geq1$. If $\mathcal{I}_C$ is
$\nu_{\alpha,\beta}$-semistable, then $E=\mathcal{I}_C$ by our
definition. Hence $\ch_0(E)=1$.

Now we assume that $\mathcal{I}_C$ is not
$\nu_{\alpha,\beta}$-semistable. One deduces

$$\nu_{\alpha,\beta}(E)=\frac{\ch^{\beta}_2(E)-\frac{1}{2}\alpha^2\ch_0(E)}{\ch^{\beta}_1(E)}
>\nu_{\alpha,\beta}(\mathcal{I}_C)=\frac{\frac{1}{2}(\beta^2-\alpha^2)-d}{-\beta},$$

i.e.,
\begin{equation}\label{eq3.1}
\ch^{\beta}_2(E)>\frac{\frac{1}{2}(\beta^2-\alpha^2)-d}{-\beta}\ch^{\beta}_1(E)+\frac{1}{2}\alpha^2\ch_0(E).
\end{equation}
By Theorem \ref{2.6}, we obtain
\begin{equation}\label{eq3.2}
\frac{\left(\ch^{\beta}_1(E)\right)^2}{2\ch_0(E)}\geq\ch^{\beta}_2(E).
\end{equation}
Combining (\ref{eq3.1}) and (\ref{eq3.2}), one sees that
$$\alpha^2\left(\ch_0(E)\right)^2+\frac{\beta^2-\alpha^2-2d}{-\beta}\ch^{\beta}_1(E)\ch_0(E)-\left(\ch^{\beta}_1(E)\right)^2<0.$$
This implies
\begin{equation}\label{eq3.3}
\ch_0(E)<\left(\frac{\beta^2-\alpha^2-2d}{\beta}+\sqrt{\Big(\frac{\beta^2-\alpha^2-2d}{\beta}\Big)^2+4\alpha^2}\right)\frac{\ch^{\beta}_1(E)}{2\alpha^2}.
\end{equation}
Since $E$ is a subobject of $\mathcal{I}_C$ in
$\Coh^{\beta}(\mathbb{P}^3)$, by the definition of
$\Coh^{\beta}(\mathbb{P}^3)$, we deduce that
$$0<\ch^{\beta}_1(E)\leq\ch^{\beta}_1(\mathcal{I}_C)=-\beta.$$
From (\ref{eq3.3}), it follows that
\begin{equation}\label{eq3.4}
\ch_0(E)<\frac{(\alpha^2-\beta^2+2d)+\sqrt{(\beta^2-\alpha^2-2d)^2+4\alpha^2\beta^2}}{2\alpha^2}.
\end{equation}
On the other hand, since $2\alpha^2+\beta^2\geq4d$, a direct
computation shows
$$\frac{(\alpha^2-\beta^2+2d)+\sqrt{(\beta^2-\alpha^2-2d)^2+4\alpha^2\beta^2}}{2\alpha^2}\leq2.$$
Therefore, by (\ref{eq3.4}), we conclude that $\ch_0(E)<2$, i.e.,
$\ch_0(E)=1$.
\end{proof}

We now compute the walls of $\mathcal{I}_C$. See \cite{Maci} for the
surface case.
\begin{lemma}\label{lemma2}
Let $E$ be a subobject of $\mathcal{I}_C$ in
$\Coh^{\beta}(\mathbb{P}^3)$ with $$(\ch_0(E), \ch_1(E),
\ch_2(E))=(r, \theta, c).$$ Then
$\nu_{\alpha,\beta}(E)\leq(<)\nu_{\alpha,\beta}(\mathcal{I}_C)$ if
and only if
$$\frac{\theta}{2}(\alpha^2+\beta^2)-(c+rd)\beta+\theta d\leq(<)0.$$
\end{lemma}
\begin{proof}
Since $E$ is a subobject of $\mathcal{I}_C$ in
$\Coh^{\beta}(\mathbb{P}^3)$, one has
$$0<\ch^{\beta}_1(E)=\theta-r\beta\leq
\ch^{\beta}_1(\mathcal{I}_C)=-\beta,$$ i.e.,
$r\beta<\theta\leq(r-1)\beta\leq0$.

Hence
$$\nu_{\alpha,\beta}(E)=\frac{\frac{r}{2}(\beta^2-\alpha^2)-\beta\theta+c}{\theta-r\beta}
\leq(<)\nu_{\alpha,\beta}(\mathcal{I}_C)=\frac{\frac{1}{2}(\beta^2-\alpha^2)-d}{-\beta}$$
is equivalent to
$$-\beta\left(\frac{r}{2}(\beta^2-\alpha^2)-\beta\theta+c\right)\leq(<)(\theta-r\beta)(\frac{1}{2}(\beta^2-\alpha^2)-d),$$
i.e.,$$\frac{\theta}{2}(\alpha^2+\beta^2)-(c+rd)\beta+\theta
d\leq(<)0.$$
\end{proof}

\begin{proposition}\label{prop1}
If $k^2<d$, then $\mathcal{I}_C$ is $\nu_{\alpha, \beta}$-semistable
for any $\alpha>0$ and $\beta=-\frac{2d}{k}$.
\end{proposition}
\begin{proof}
We let $\alpha_0$ be an arbitrary positive real number,
$\beta_0=-\frac{2d}{k}$, and let $E$ be the
$\nu_{\alpha_0,\beta_0}$-maximal subobject of
$\mathcal{I}_C\in\Coh^{\beta_0}(\mathbb{P}^3)$.

Since $k^2<d$, one sees that $2\alpha_0^2+\beta_0^2>\beta_0^2>4d$.
Hence, by Lemma \ref{lemma1}, one has $\ch_0(E)=1$, and $E$ is
subsheaf of $\mathcal{I}_C$. We can write $E=\mathcal{I}_W(-l)$,
where $W\subset \mathbb{P}^3$ is a scheme of dimension $\leq 1$ and
$l\geq0$. The Chern characters of $\mathcal{I}_W(-l)$ are
$$(\ch_0(\mathcal{I}_W(-l)),\ch_1(\mathcal{I}_W(-l)),\ch_2(\mathcal{I}_W(-l)))=(1, -l,\frac{1}{2}l^2+\ch_2(\mathcal{I}_W)).$$
Since $\mathcal{I}_W(-l)$ is a subobject of $\mathcal{I}_C$ in
$\Coh^{\beta_0}(\mathbb{P}^3)$, one deduce
$$0<\ch^{\beta_0}_1(\mathcal{I}_W(-l))=-l-\beta_0\leq
\ch^{\beta_0}_1(\mathcal{I}_C)=-\beta_0,$$ i.e.,
\begin{equation}\label{eq3.5}
0\leq l<-\beta_0.
\end{equation}

If $C\subseteq W$, then $\ch_2(\mathcal{I}_W)\leq
\ch_2(\mathcal{I}_C)=-d$. Thus one sees that

\begin{eqnarray*}
\frac{-l}{2}(\alpha_0^2+\beta_0^2)-(\frac{1}{2}l^2+\ch_2(I_W)+d)\beta_0-ld&\leq&\frac{-l}{2}\beta_0^2-(\frac{1}{2}l^2-d+d)\beta_0\\
&=&\frac{-\beta_0l}{2}(l+\beta_0)\\
&\leq&0.
\end{eqnarray*}
By Lemma \ref{lemma2}, we conclude that
$\nu_{\alpha_0,\beta_0}(\mathcal{I}_W(-l))\leq\nu_{\alpha_0,\beta_0}(\mathcal{I}_C)$.
Therefore the $\nu_{\alpha_0,\beta_0}$-maximal subobject of
$\mathcal{I}_C$ in $\Coh^{\beta_0}(\mathbb{P}^3)$ is $\mathcal{I}_C$
itself. Namely, $\mathcal{I}_C$ is $\nu_{\alpha_0,
\beta_0}$-semistable.

If $C\nsubseteq W$, then $\mathcal{I}_W(-l)\subset\mathcal{I}_C$
implies $\mathcal{O}_{\mathbb{P}^3}(-l)\subset\mathcal{I}_C$. Thus
$l\geq k$. One deduces by (\ref{eq3.5}) that
\begin{eqnarray}\label{eq3.6}
\frac{-l}{2}(\alpha_0^2+\beta_0^2)-(\frac{1}{2}l^2+\ch_2(I_W)+d)\beta_0-ld&<&\frac{-l}{2}\beta_0^2-(\frac{1}{2}l^2+d)\beta_0-ld
\nonumber
\\
&=&\frac{-l}{2}\left(\beta^2_0+(l+\frac{2d}{l})\beta_0+2d\right)\nonumber
\\
&=&\frac{-l}{2}(\beta_0+l)(\beta_0+\frac{2d}{l})\\
&=&\frac{-l}{2}(\beta_0+l)(\frac{2d}{l}-\frac{2d}{k})\nonumber
\\
&\leq&\nonumber 0.
\end{eqnarray}
From Lemma \ref{lemma2}, it follows that $\mathcal{I}_C$ is also
$\nu_{\alpha_0, \beta_0}$-semistable in this case.
\end{proof}

\begin{proposition}\label{prop2}
If $k^2\geq d$, then $\mathcal{I}_C$ is $\nu_{\alpha,
\beta}$-semistable for any $\alpha>0$ and $\beta=-2\sqrt{d}$.
\end{proposition}
\begin{proof}
The proof is almost the same as that of Proposition \ref{prop1}. We
let $\alpha_0$ be an arbitrary positive real number,
$\beta_0=-2\sqrt{d}$, and let $E$ be the
$\nu_{\alpha_0,\beta_0}$-maximal subobject of
$\mathcal{I}_C\in\Coh^{\beta_0}(\mathbb{P}^3)$.

By Lemma \ref{lemma1}, the assumption $\beta_0=-2\sqrt{d}$ makes
sure that $\ch_0(E)=1$. We can still write $E=\mathcal{I}_W(-l)$ as
in the proof of Proposition \ref{prop1}. When $C\subseteq W$, the
same proof of Proposition \ref{prop1} shows that $\mathcal{I}_C$ is
$\nu_{\alpha_0, \beta_0}$-semistable.

In the case of $C\nsubseteq W$, one sees that $l\geq k$. Thus it
follows from (\ref{eq3.6}) and (\ref{eq3.5}) that
\begin{eqnarray*}
\frac{-l}{2}(\alpha_0^2+\beta_0^2)-(\frac{1}{2}l^2+\ch_2(I_W)+d)\beta_0-ld&<&
\frac{-l}{2}(\beta_0+l)(\beta_0+\frac{2d}{l})\\
&\leq&\frac{-l}{2}(\beta_0+l)(\frac{2d}{k}-2\sqrt{d}).
\end{eqnarray*}
The assumption $k^2\geq d$ guarantees that the left hand side of the
above inequality is negative. Therefore we are done by Lemma
\ref{lemma2}.
\end{proof}

\section{The proof of the main theorems}\label{S4}
Now we can prove Theorem \ref{main} and \ref{main2} easily.
\begin{proof}[Proof of Theorem \ref{main}]
Since $C$ is an integral curve, one sees that
$$\ch^{\beta}_3(\mathcal{I}_C)=-\frac{1}{6}\beta^3+d\beta+2d-\chi(\mathcal{O}_C).$$
If $\mathcal{I}_C$ is $\nu_{\alpha,\beta}$-semistable, then from
Theorem \ref{*}, it follows that
\begin{eqnarray*}
&&\alpha^2\overline{\Delta}^{\beta}_H(\mathcal{I}_C)+4\left(H\ch^{\beta}_2(\mathcal{I}_C)\right)^2-6H^2\ch^{\beta}_1(\mathcal{I}_C)\ch^{\beta}_3(\mathcal{I}_C)\\
=&&2\alpha^2d+4d^2+\beta^4-4\beta^2d-6(-\beta)\left(-\frac{1}{6}\beta^3+d\beta+2d-\chi(\mathcal{O}_C)\right)\\
=&&2\alpha^2d+4d^2+2\beta^2d+6\beta(2d-\chi(\mathcal{O}_C))\\
\geq&&0,
\end{eqnarray*}
i.e.,
\begin{equation}\label{eq4}
h^1(\mathcal{O}_C)-1=-\chi(\mathcal{O}_C)\leq\frac{2d^2+(\alpha^2+\beta^2)d}{3(-\beta)}-2d.
\end{equation}

By Proposition \ref{prop1} and \ref{prop2}, one can substitute
$(\alpha, \beta)=(0, -\frac{2d}{k})$ and $(\alpha, \beta)=(0,
-2\sqrt{d})$ into (\ref{eq4}) respectively to obtain our desired
conclusion.
\end{proof}

\begin{proof}[Proof of Theorem \ref{main2}]
The short exact sequence
$$0\rightarrow \mathcal{I}_C(m)\rightarrow \mathcal{O}_{\mathbb{P}^3}(m)\rightarrow
\mathcal{O}_C(m)\rightarrow0$$ induces a long exact sequence
$$H^1(\mathcal{O}_{\mathbb{P}^3}(m))\rightarrow H^1(\mathcal{O}_C(m))\rightarrow H^2(\mathcal{I}_C(m))\rightarrow H^2(\mathcal{O}_{\mathbb{P}^3}(m)).$$
Since
$H^1(\mathcal{O}_{\mathbb{P}^3}(m))=H^2(\mathcal{O}_{\mathbb{P}^3}(m))=0$,
we deduce $h^2(\mathcal{I}_C(m))=h^1(\mathcal{O}_C(m))$.

Now we assume that
\begin{assumption}\label{ass}
$m>\frac{2d}{k}$, $k^2<d$ and $\beta_0=-\frac{2d}{k}$.
\end{assumption}
One sees that
$$\ch^{\beta_0}_1(\mathcal{O}_{\mathbb{P}^3}(-m))=-m+\frac{2d}{k}<0.$$
Thus
$\mathcal{O}_{\mathbb{P}^3}(-m)[1]\in\Coh^{\beta_0}(\mathbb{P}^3)$.
It turns out that
$$\nu_{\alpha_0,\beta_0}(\mathcal{O}_{\mathbb{P}^3}(-m)[1])=\frac{-\frac{1}{2}(m+\beta_0)^2+\frac{1}{2}\alpha_0^2}{m+\beta_0}
<\nu_{\alpha_0,\beta_0}(\mathcal{I}_C)=\frac{\frac{1}{2}(\beta_0^2-\alpha_0^2)-d}{-\beta_0}$$
is equivalent to
$$-\beta_0\left(-\frac{1}{2}(m+\beta_0)^2+\frac{1}{2}\alpha_0^2\right)<(m+\beta_0)(\frac{1}{2}(\beta_0^2-\alpha_0^2)-d),$$
i.e.,$$\alpha_0^2+\beta_0^2+(m+\frac{2d}{m})\beta_0+2d<0.$$

Assumption \ref{ass} implies
\begin{eqnarray*}
\beta_0^2+(m+\frac{2d}{m})\beta_0+2d&=&(\beta_0+m)(\beta_0+\frac{2d}{m})\\
&=&(\beta_0+m)(\frac{2d}{m}-\frac{2d}{k})\\
&<&(\beta_0+m)(k-\frac{2d}{k})\\
&<&0.
\end{eqnarray*}
Thus we can find an $\alpha_0>0$ such that
$\nu_{\alpha_0,\beta_0}(\mathcal{O}_{\mathbb{P}^3}(-m)[1])<\nu_{\alpha_0,\beta_0}(\mathcal{I}_C)$.
On the other hand, by \cite[Proposition 7.4.1]{BMT} and Proposition
\ref{prop1}, one deduces that $\mathcal{O}_{\mathbb{P}^3}(-m)[1]$
and $\mathcal{I}_C$ are both $\nu_{\alpha_0,\beta_0}$-semistable. We
conclude that $$\Hom_{\D^b(\mathbb{P}^3)}(\mathcal{I}_C,
\mathcal{O}_{\mathbb{P}^3}(-m)[1])=0.$$ By the Serre duality
theorem, one obtains $h^2(\mathcal{I}_C(m-4))=0$. Therefore we
conclude that $h^2(\mathcal{I}_C(l))=h^1(\mathcal{O}_C(l))=0$ if
$l>\frac{2d}{k}-4$ and $k^2<d$.

Similarly, one can show
$h^2(\mathcal{I}_C(l))=h^1(\mathcal{O}_C(l))=0$ if
 $l>2\sqrt{d}-4$ and $k^2\geq d$.
\end{proof}

\bibliographystyle{amsplain}

\end{document}